\documentclass{article}
\usepackage[fontsize=12pt]{scrextend}
\usepackage{babel}
\usepackage{amsmath}
\usepackage{mathtools}
\usepackage{euscript, amsmath,amssymb,amsfonts,mathrsfs,amsthm,mathtools,graphicx, tikz, xcolor,verbatim, bm, enumerate, enumitem,multicol,appendix,etoolbox}
\usepackage{wrapfig}
\usepackage[all]{xy}
\usepackage{upquote}
\usepackage{listings}
\usetikzlibrary{arrows,patterns}
\usepackage{authblk}
\usepackage{caption}
\allowdisplaybreaks

\usepackage[latin1]{inputenc}
\usepackage{verbatim}
\usepackage{bm}
\usepackage[justification=centering]{subcaption}

\lstdefinelanguage{Sage}[]{Python}
{morekeywords={True,False,sage,singular},
sensitive=true}
\definecolor{dblackcolor}{rgb}{0.0,0.0,0.0}
\definecolor{dbluecolor}{rgb}{.01,.02,0.7}
\definecolor{dredcolor}{rgb}{0.8,0,0}
\definecolor{dgraycolor}{rgb}{0.30, 0.3,0.30}

\usepackage[outer=1in,marginparwidth=.75in]{geometry}
\usepackage{marginnote}
\usetikzlibrary{calc}
\usetikzlibrary{positioning}
\usetikzlibrary{shapes.geometric}
\usetikzlibrary{shapes.geometric}

\usepackage{color}
\usepackage[latin1]{inputenc}
\usepackage{wrapfig}

\tikzstyle{square} = [shape=regular polygon, regular polygon sides=4, minimum size=1cm, draw, inner sep=0, anchor=south, fill=gray!30]
\tikzstyle{squared} = [shape=regular polygon, regular polygon sides=4, minimum size=1cm, draw, inner sep=0, anchor=south, fill=gray!60]

\newtheorem{theorem}{Theorem}[section]
\newtheorem{definition}[theorem]{Definition}

\newtheorem{lemma}[theorem]{Lemma}
\newtheorem{coro}[theorem]{Corollary}

\newcommand{\Z}{{\mathbb{Z}}}

\newcommand{\Nm}{{\mathrm{Nm}}}

\providecommand{\keywords}[1]
{
  \small	
  \textbf{\textit{Keywords---}} #1
}

\begin{document}

\title{Sizes of Pre-Images of the Minimal Euclidean Function on the Gaussian Integers}
\author{Hester Graves}
\affil{Center for Computing Sciences/IDA}
\date{\today}

\maketitle

\abstract{
In 2023, the author presented the first computable minimal Euclidean function for a non-trivial number field.
Along with a formula for $\phi_{\Z[i]}$, the minimal Euclidean function on the Gaussian inteers, the same paper 
introduced a geometric description for $\phi_{\Z[i]}^{-1}([0,n])$.
This paper uses that construction to prove formulas for the size of the function's pre-images, or $|\phi_{\Z[i]}^{-1}([0,n])|$.


\section{Introduction}

Motzkin showed every Euclidean domain $R$ has a minimal Euclidean function, $\phi_R$ \cite{Motzkin}.
Gauss demonstrated the absolute norm is a Euclidean function on the domain $\Z[i] = \{x +yi :a, b \in \Z\}$, the so-called Gaussian integers.
In other words, for all non-zero $a,b \in \Z[i]$, there exist some quotient and remainder $q$ and $ r \in \Z[i]$ such that $a = qb +r$, where either 
$r =0$  or $|\Nm(r)| < |\Nm(b)|$.
The absolute norm is a Euclidean function on $\Z[i]$, but it is not the minimal Euclidean function, $\phi_{\Z[i]}$.

Since $2 = -i(1+i)^2$ and all integers have binary expansions, every Gaussian integer $x +yi$ has an $(1+i)$-ary expansion, $x +yi = \sum_{j=0}^N u_j (1+i)^j$, 
where the $u_j$ are multiplicative units, i.e., the $u_j \in \Z[i]^{\times} = \{ \pm 1, \pm i\}$.
The easiest way to find a $(1+i)$-ary expansion for $x +yi = (x-y) +y(1+i)$ is to take the binary expansions of $(x -y)$ and $y$ and rewrite them as sums of units times  powers of $1+i$.
 H W Lenstra Jr showed in Section 11 of \cite{Lenstra} that 
$$\phi_{\Z[i]}(x +yi) = \min \left \{n : \exists u_j \in \Z[i]^{\times} \text{ such that } \sum_{j=0}^n u_j (1+i)^j = x +yi \right \}.$$

Samuel described the pre-images $\phi_{\Z[i]}^{-1}([0,n])$ as `very irregular' (p 290, \cite{Samuel}), but in his dissertation, Simon Plouffe \cite{Plouffe} conjectured a formula for $a(n+1) = |\phi_{\Z[i]}^{-1} ([0,n])|$.
Robert Israel computed a large table of values of $a(n)$, and found formulas for $a(n)$, both published in the OEIS' entry for the sequence \cite{oeis}.
He claimed
\begin{align*}
a(2k) &= 14\cdot 4^k - 34\cdot 2^k + 8\cdot k + 21\\
a(2k+1)&= 28\cdot 4^k - 48\cdot 2^k + 8\cdot k + 25.
\end{align*}
Interestingly, the entry also links to a letter from Lenstra to N J A Sloane, where he also computed early entries in this sequence.

In \cite{Graves}, the author found a formula to explicitly compute $\phi_{\Z[i]}$, as well as a new geometric description of the set $\phi_{\Z[i]}^{-1}([0,n])$.
This work applies that geometric description to give straightforward proofs of Israel's formulas.
The author knows of no other published proof of Israel's work.

\section{Background}

\begin{minipage}[t]{0.5\textwidth}
\vspace{0pt}
The author previously showed that $\phi_{\Z[i]}^{-1}([0,n])$ forms a perforated octagon in the plane.  
Octagons are a special case of the following definition.

\begin{definition} Suppose $a$ and $b$ are integers, where $0 < a < b <2a$.  
We define 
$$E_{a,b} := \{x +yi: |x|, |y| \leq a, |x| + |y| \leq b\}.$$
\end{definition}

All of the geometric objects studied are defined in terms of the following (increasing) 
sequence.
\end{minipage}
\hfill 
\begin{minipage}[t]{0.40\textwidth}
\vspace{0pt}
\centering
\includegraphics{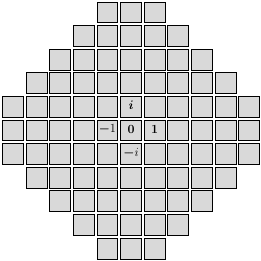}
\captionof{figure}{$E_{5,6}$}
\end{minipage}

\begin{definition} For $n \geq 0$, let 
$$w_n := \begin{cases}
3 \cdot 2^k \text{ if }n =2k.\\
4 \cdot 2^k \text{ if } n = 2k +1
\end{cases}
.$$
\end{definition}
\textbf{Example}: The sequence begins $3, 4,6,8,12,16,24,32,48,64,\ldots.$\\

If $a$ divides $b$, we write $a | b$.  Otherwise, we write $a \nmid b$.  
We denote the greatest common divisor (or highest common factor) of $a$ and $b$ by $\gcd(a,b)$.
\begin{definition}
For $n \geq 0$, let $Oct_n = E_{w_n -2, w_{n+1} -3}$ and let 
$$S_n := \{ x +yi \in Oct_n: 2 \nmid \gcd(x,y)\}.$$
\end{definition}

\begin{figure}[ht]
\centering
\includegraphics{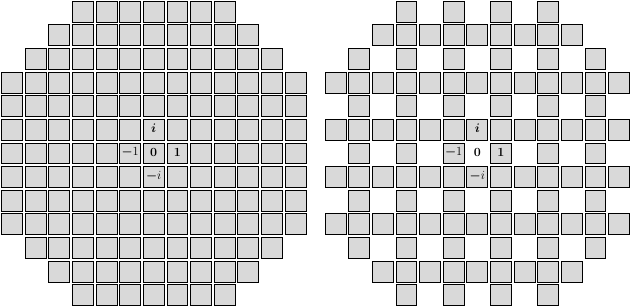}
\caption{$Oct_3$ and $S_3$}
\end{figure}

These seemingly unmotivated sets are important due to the following theorem.
We use the disjoint union symbol because the sets are, indeed, disjoint.\\
\begin{minipage}[t]{0.5\textwidth}
\vspace{0pt}
\begin{theorem} (Theorem 2.2, \cite{Graves})
For $n \geq0$, 
$$\phi_{\Z[i]}^{-1}([0,n]) \setminus \{0\}= \bigsqcup_{j=0}^{\lfloor n/2 \rfloor} 2^j S_{n-2j}.$$
\end{theorem}
\begin{coro} \label{coro:adding_up}For $n \geq 0$, 
$$a_{n+1}=|\phi_{\Z[i]}^{-1}([0,n]) | = 1 + \sum_{j=0}^{\lfloor n/2 \rfloor} |S_{n-2j}|.$$
\end{coro}
\end{minipage}
\hspace{1em}
\begin{minipage}[t]{0.4\textwidth}
\vspace{0pt}
\centering
\includegraphics{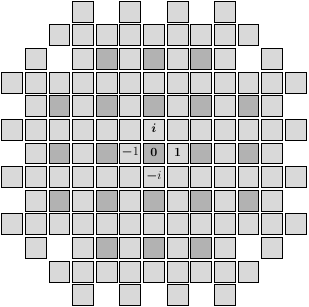}
\captionof{figure}{$\phi_{\Z[i]}^{-1}([0,3]) = S_3 \cup 2 S_1$\\
\footnotesize{
\hspace{5em}$S_3$ is in light gray\\
\hspace{5em}$2S_1$ is in darker gray}}
\end{minipage}

\section{Counting Building Blocks}

In order to find a formula for $|S_n|$, we first find a general formula for $|E_{a,b}|$.

\begin{lemma}\label{size_of_E} The number of points in $E_{a,b}$ is $1 + 4a + 4a^2 - 2(2a-b)(2a-b +1)$.
\end{lemma}
\begin{proof}
Figure \ref{fig:four_pieces} show the number of points in $E_{a,b}$ is 
\begin{align*}
\sum_{|x| \leq a} 
\sum_{
\tiny \begin{array}{c}|y| \leq a \\ 
|x| + |y| \leq b
\end{array}
} 1 
&= 1 + 4 \sum_{0 \leq x \leq a} 
\hspace{-.25 cm}
\sum_{
\tiny \begin{array}{c}
1\leq y \leq a\\1 \leq y \leq b-x
\end{array}} 
\hspace{-.5 cm} 1\\
&= 1 + 4a + 4 \sum_{1 \leq x \leq a}
\hspace{-.25 cm} \sum_{
\tiny
\begin{array}{c}
1 \leq y \leq a\\
1 \leq y \leq b -x.
\end{array}}
\hspace{-.5 cm} 1.\\
\intertext{The last sum represents the points in an $a \times a$ square with a right isoceles triangle removed from
the upper right corner.  
The triangle side length is $a - (b-a) =2a -b$, so the number of points in the triangle is the 
$(2a-b)$'th trianglular number, $T_{2a -b} = \frac{(2a-b)(2a-b +1)}{2}$.
We can therefore rewrite our expression as}
&= 1 + 4a + 4\left (a^2 -  \frac{(2a-b)(2a-b +1)}{2} \right )\\
&= 1 + 4a + 4a^2 - 2(2a -b)(2a -b +1). \qedhere
\end{align*}
\end{proof}

\begin{figure}[ht]
\centering
\includegraphics{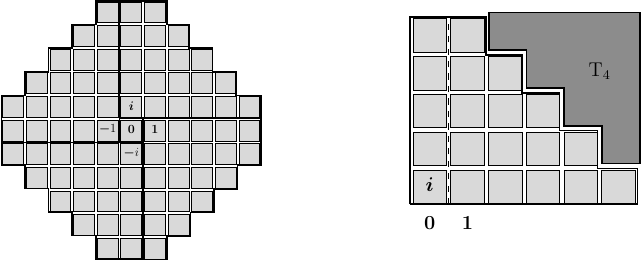}
\caption{$E_{5,6}$ divided into four pieces,
 which are unit multiples
 of each other\\
 Each piece consists of a column, as well as a $5 \times 5$ square with a right isoceles\\
  triangle of side length $4$ bitten out of the top corner}
\label{fig:four_pieces}
\end{figure}

\begin{coro}\label{oct_coro} For $n \geq 0$, 
$$|Oct_n | = 1 +4(w_n -2) + 4(w_n -2)^2 - 2(w_{n+2} - w_{n+1})(w_{n+2} - w_{n+1} -1).$$
\end{coro} 
\noindent \textbf{Examples:} We see $|Oct_0| =5$, $|Oct_1| =21$, $|Oct_2| =57$, and $|Oct_3| = 145$.\\

\begin{lemma}\label{even_points} For $n \geq 1$, the number of points $x +yi \in Oct_n$ where $2 | \gcd(x,y)$ is 
$$1 + 2(w_n -2) + (w_n -2)^2 - 2(w_n - w_{n-1})(w_n - w_{n-1} +1).$$
\end{lemma}
\begin{proof}
The number of points in $Oct_n$ with even real and imaginary parts is 
\begin{align*}\sum_{\tiny
|2l| \leq w_n -2}
\hspace{-.25 cm}
\sum_{\tiny
\begin{array}{c}
|2j| \leq w_n -2\\
|2l| + |2j| \leq w_{n+1} -3
\end{array}} 1.&
\intertext{Since $2j$, $2l$, and $w_{n+1}$ are all even, 
$|2l| + |2j| \leq w_{n=1} -3$ implies $|2l| + |2j| \leq w_{n+1} -4$.
We can therefore rewrite the sum as }
\sum_{\tiny
|2l| \leq w_n -2}
\sum_{\tiny
\begin{array}{c}
|2j| \leq w_n -2\\
|2l| + |2j| \leq w_{n+1} -4
\end{array}
} 1 =
&\sum_{\tiny |l| \leq w_{n-2} -1}
\sum_{\tiny
\begin{array}{c}
|j| \leq w_{n-2} -1\\
|l| + |j| \leq w_{n-1} -2
\end{array}
} 1 =
E_{w_{n-2} -1, w_{n-1} -2}
.
\end{align*}
We apply Lemma \ref{size_of_E} and simplify, using $w_n = 2 w_{n-2}$, to prove the claim.
\end{proof} 

\begin{lemma}\label{size_S}
For $n\geq 0$, 
\[|S_n| =2(w_n -2) + 3(w_n -2)^2 - 6(w_n - w_{n-1})(w_n - w_{n-1} -1).\]
\end{lemma}
\begin{proof}
The elements of $S_n$ are the points in $x +yi \in Oct_n$ such that $2 \nmid \gcd(x,y)$, so by 
Corollay \ref{oct_coro}, Lemma \ref{even_points}, and the identity $w_{n+2} = 2 w_n$,
\begin{align*}
|S_n|  = & [ 1 +4(w_n -2) + (4(w_n -2)^2 + 2(w_{n+2} - w_{n+1})(w_{n+2} - w_{n+1} -1)]\\
& - [1 + 2(w_n -2) + (w_n -2)^2 - 2(w_n - w_{n-1})(w_n - w_{n-1} +1)]\\
=& 2(w_n -2) + 3 (w_n -2)^2 + \\
& 2(w_{n+2} - w_{n+1})(w_{n+2} - w_{n+1} -1) - (w_{n+2} - w_{n+1})(w_n - w_{n-1} +1)\\
= & 2(w_n -2) + 3 (w_n -2)^2 + \\
&(w_{n+2} - w_{n+1}) [2(w_{n+2} - w_{n+1} -1) - (w_n - w_{n-1} +1)]\\
= & 2(w_n -2) + 3 (w_n -2)^2 + \\
&(w_{n+2} - w_{n+1}) [4w_{n} - 4w_{n-1} -2 - w_n + w_{n-1} -1]\\
= & 2(w_n -2) + 3 (w_n -2)^2 +
3(w_{n+2} - w_{n+1}) (w_{n} - w_{n-1} -1)\\
= & 2(w_n -2) + 3 (w_n -2)^2 +
6(w_{n} - w_{n-1}) (w_{n} - w_{n-1} -1). \qedhere
\end{align*} 
\end{proof}
\noindent
\textbf{Examples:} Direct computation shows $|S_0| =4$.  After that, the lemma gives us the sequence
$|S_1| = 16$, $|S_2| =44$, and $|S_3| = 108$.  

\section{Summing Up the Main Results}
These sums are suprisingly easy, in that they only requiring knowing how to add up geometric series.

\begin{theorem} For $k\geq 1$, $|\phi_{\Z[i]}^{-1} ([0,2k])| = 25 + 8k - 48\cdot 2^k + 28 \cdot 4^k.$
\end{theorem}
\begin{proof} We plug $|S_0| = 4$, $w_n = w_{2k} = 3 \cdot 2^k$, and $w_{n-1} = w_{2k-1} = 2 \cdot 2^k$ into the expressions from 
Corollary \ref{coro:adding_up} and Lemma \ref{size_S} to see 
\begin{align*}
| \phi_{\Z[i]}^{-1}([0,2k])| &=  1 + |S_0| + \sum_{j=1}^k |S_{2j}|\\
&= 1 + 4 + \sum_{j=1}^k \left [ 2(3 \cdot 2^j -2) + 3(3 \cdot 2^j -2)^2 - 6(2^j)(2^j -1) \right ]\\
& = 5 + \sum_{j=1}^k \left [ 6 \cdot 2^j -4 + 27 \cdot 4^j - 36 \cdot 2^j +12 - 6 \cdot 4^j + 6 \cdot 2^j \right ]\\
&= 5 + 8k - 24(2^{k+1} -2) + 21 \left ( \frac{4^{k+1} -4}{3} \right )\\
&=25 + 8k - 48 \cdot 2^k + 28 \cdot 4^k. \qedhere
\end{align*}
\end{proof}

\begin{theorem} For $k \geq 1$, $|\phi_{\Z[i]}^{-1} ([0, 2k +1])| = 29 + 8k - 68 \cdot 2^k + 28 \cdot 4^k.$
\end{theorem}
\begin{proof}
Using the identities $w_n = w_{2k +1} = 4 \cdot 2^k$ and $w_{n-1} = w_{2k} = 3 \cdot 2^k$, as well as Corollary \ref{coro:adding_up} and Lemma \ref{size_S}, shows
\begin{align*}
|\phi_{\Z[i]}^{-1}([0, 2k+1])| & =1 + \sum_{j=0}^k |S_{2j +1}|\\
&= 1 + \sum_{j=0}^k 
[ 2(4 \cdot 2^j) + 3(4 \cdot 2^j -2)^2 - 6(2^j)(2^j -2)]\\
&= 1 + \sum_{j=0}^k
[8 \cdot 2^j -4 + 48 \cdot 4^j - 48 \cdot 2^j +12 - 6 \cdot 4^j + 6 \cdot 2^j ]\\
&= 1 + \sum_{j=0}^k[ 8 - 34 \cdot 2^j + 42 \cdot 4^j]\\
&= 1 + (8k+1) - 34 (2^{k+1} -1) + 42 \left ( \frac{4^{k+1} -1}{3} \right )\\
&=9 +8k - 68 \cdot 2^k + 34 + 14 \cdot 4^{k+1} -14\\
&= 29 + 8k - 68 \cdot 2^k + 28 \cdot 4^k.
\qedhere
\end{align*}
\end{proof}

Note that, according to Robert Israel's formula,
$|\phi_{\Z[i]}^{-1} ([0, 2k +1]) | = a(2k +2) = a(2(k+1))$, so our formulas match.

\section{Acknowledgements}
I would like to thank my colleague Tad White, who pointed out the OEIS entry A006457 and its references.

\end{document}